\documentclass[12pt,reqno]{amsart}
\usepackage[margin=1.0in]{geometry} 
\usepackage{amsmath,amssymb,amsthm,mathtools}
\usepackage[T1]{fontenc}
\usepackage{lmodern}
\usepackage[colorlinks=true,citecolor=blue,linkcolor=blue,urlcolor=blue]{hyperref}
\usepackage{comment}
\newtheorem{theorem}{Theorem}[section]

\newtheorem{lemma}[theorem]{Lemma}
\newtheorem{corollary}[theorem]{Corollary}
\theoremstyle{remark}

\numberwithin{equation}{section}

\numberwithin{equation}{section}

\newcommand{\Zp}{\mathbb{Z}_p}

\title[Multiplicative subgroups that are GAP's]{Multiplicative Subgroups of $\Zp^*$ that are\\ Generalized Arithmetic Progressions}

\author[Albert Cochrane]{Albert Cochrane}
\address{Dept. of Mathematics \\ Kansas State University\\
Manhattan, KS 66506}
\email{albertc@ksu.edu}

\begin{document}

\begin{abstract}
We prove that a multiplicative subgroup $A_k$ of $\Zp^*$ is a generalized arithmetic progression if and only if $|A_k| = 2,\ 4,$ or $p-1$. Much of the argument builds upon recent work studying additive decompositions of subgroups, and we generalize a result of Hanson and Petridis to show that any additive $n$-decomposition of a subgroup must be a direct sum. We also show how this classification quickly follows from Kalmynin's recent work resolving S\'ark\"ozy's conjecture for quadratic residues.
\end{abstract}

\maketitle

\section{Introduction}
Our primary goal in this paper is to completely answer the question of when multiplicative subgroups can have the additive structure of a generalized arithmetic progression (GAP) in the finite field of $p$ elements. Let $p$ be a prime number, $\Zp$ be the finite field of integers modulo $p$, and $\Zp^*$ denote its multiplicative group of units. For a divisor $k$ of $p-1$ and $t = \frac{p-1}{k}$, we denote the multiplicative subgroup of $\Zp^*$ with order $t$ as $A_k = \{x^k : x \in \Zp^* \}$.  A GAP of dimension $n$ in $\mathbb Z_p$ is a set of the form
\[ 
S=\{a+x_1 d_1 + \dots + x_n d_n : x_i \in [0,L_i-1], 1 \leq i \leq n\},
\]
where $d_i \in \mathbb Z_p^*$ and $L_i \geq 2$ for all $1 \leq i \leq n$. We say that $S$ is a \textit{proper} generalized arithmetic progression if $|S| = \prod_{i=1}^n L_i$.  

For $n=1$, Chowla, Mann, and Straus \cite{cho} proved that a subgroup $A_k$ is an arithmetic progression if and only if $t = 1, 2$ or $p-1$.  Moving on to the case of $n = 2$, it is easy to see that any subgroup of order $4$ is a proper two dimensional GAP. Let $p$ be a prime congruent to $1 \pmod{4}$ and set  $k = \frac{p-1}{4}$. Then for a generator $\omega$, the subgroup $A_k = \{ \pm 1, \pm \omega \}$, which can be written as:
\[
A_k = \{-1 + x_1(1-\omega) + x_2(1+\omega): x_1, x_2 \in [0,1] \}.
\]
We claim that over all prime numbers $p$, there are no other cases where a subgroup of $\Zp^*$ can be a GAP.

\begin{theorem}\label{mainthm}
Let $A_k$ be a multiplicative subgroup of $\Zp^*$ with $|A_k| = t$.  Then, $A_k$ is a generalized arithmetic progression if and only if $t = 2,4,$ or $p-1$. 
\end{theorem}

\section{Definitions and Notation}

We use several standard definitions from additive combinatorics.  Letting $A,B \subseteq \Zp$, we define the sumset $A+B = \{a + b : a \in A, b \in B\}$. For a positive integer $l \geq 2$, we denote the $l$-fold sumset as $lA = A + \dots + A = \{a_1 + \dots + a_l : a_i \in A \}$.  We note here that GAP's of dimension $1$ will simply be referred to as arithmetic progressions, with an occasional reminder that we are speaking of the $n=1$ case depending on context. The definition of GAP in the introduction follows \cite{ta} and other papers on the topic, for example \cite{cr}.  We equivalently use the shorthand  notation found in \cite{zh}, presenting the GAP structure as a sumset of arithmetic progressions:
\[
S =a +d_1[0,L_1-1]+ \dots + d_n[0, L_n -1].
\] This notation is particularly useful when considering GAP's as a special type of additive decomposition, and we may informally refer to $L_i$ as being the size of the $i$th component of the progression.

A subset $B$ of $\Zp$ admits a non-trivial additive decomposition if $B = S_1 + S_2$ for some $S_1,S_2 \subseteq \Zp$ with $|S_1|, |S_2| \geq 2$. More generally, we say $B$ admits an $n$-decomposition if there exist $S_i \subseteq \Zp$ with $|S_i| \geq 2$ such that
\[
B=S_1+\cdots +S_n.
\] Any $n$-decomposition with $n>2$ can also be expressed as a $2$-decomposition, $B = S_1+(S_2+ \cdots +S_n)$.  Following Shkredov \cite{shk2}, we say that a decomposition $B=S_1+\cdots +S_n$ is a direct sum if $|B| = \prod_{i=1}^n|S_i|$. Notice that for both GAP's and $n$-decompositions, we include the standard conditions for non-triviality in the definition statement. That is to say, in this paper we are only interested in the cases where each $L_i \geq 2$ for GAP's and $|S_i| \geq 2$ for decompositions.

\section{Background}

The question of whether a multiplicative subgroup can be a generalized arithmetic progression directly points at the sum-product phenomenon in $\Zp$. Indeed, the results in this paper illustrate that in nearly all cases the additive structure of progressions is incompatible with the multiplicative structure found in subgroups of $k$th powers. 

A closely related question which has received much attention in recent years is whether a subgroup can be expressed as a sumset, that is, admit an additive decomposition. S\'ark\"ozy \cite{sa} conjectured that for $p$ sufficiently large, the subgroup of quadratic residues $R$ cannot be expressed as a non-trivial $2$-decomposition. Many investigations into the additive irreducibility of multiplicative subgroups have followed; See for example, \cite{dart1, shk1, shp1, wu, yip1}.  A recent preprint of Kalmynin \cite{ka} resolves this conjecture for quadratic residues, and provides new insights on decompositions of arbitrary subgroups of $k$th powers which we consider in Section 6.

Specifically relevant to our main argument, Chen and Yan \cite{che} refined bounds on decompositions of $R$, proving that any additive decomposition of the form $R = U + V$ must satisfy $|V| \geq |U| \geq 5$, a result we use in Section 5. Hanson and Petridis \cite{ha} showed that any $2$-decomposition of a multiplicative subgroup must be a direct sum. This result in particular proved instrumental in allowing us to consider generalized progressions of arbitrarily large dimension. In the next section we extend this result to cover $n$-decompositions with $n >2$, and ultimately are able to restrict our attention to proper GAP's such that $L_i = 2$ uniformly.

\section{GAP's and Additive Decompositions}

By definition, any GAP admits an additive decomposition where each component is a ($1$-dimensional) arithmetic progression of size at least 2. A nice feature of arithmetic progressions is that they can always be written as a decomposition of sets with size exactly $2$. Suppose $S$ is an arithmetic progression with $L$ terms. Then for some $a \in \Zp$ and $d \in \Zp^*$, we can express $S$ as the sum of $\{a, a+d\}$ and $\{0,d \}$ added to itself $L-2$ times:
\[
S = \{a, a + d, \dots , a + (L-1) d\ \} = \{a, a+ d\} + (L-2)\{0,\ d\}.
\]

Plainly, any sumset of additive decompositions will itself be a decomposition, so this proves the following lemma. For now we use $m$ instead of $n$ to avoid conflating dimension with the number of components in the decomposition.

\begin{lemma}\label{size2lem} Suppose $S \subseteq \Zp$ is a generalized arithmetic progression. Then, for some positive integer $m$, there exist subsets $S_i \subseteq \Zp$ with $|S_i| = 2$ for $1 \leq i \leq m$ such that $S_1 + \dots + S_m = S$.
\end{lemma}

\noindent That is to say, any GAP admits an additive decomposition where each component is of size exactly 2. 

We now shift our attention to the condition of a decomposition being a direct sum. Hanson and Petridis \cite{ha} proved the following for $2$-decompositions.

\begin{lemma}[\cite{ha}; Corollary 1.3]\label{hp}
Let $p$ be a prime and $A_k$ be the multiplicative subgroup of $k$th powers for some $k | (p-1)$ with $k \neq 1$. Suppose $S ,\ T \subseteq \Zp$ satisfy $S + T = A_k$. Then, $|S| |T| = |A_k| $. 
\end{lemma}

We also need the following theorem of Gyarmati, Matolsci, and Ruzsa \cite{gy}.

\begin{theorem}[\cite{gy}; Theorem 1.2]\label{ruz}
Let $n \ge 3$ be a positive integer and $S_1,\dots, S_n$  be nonempty subsets of $\mathbb Z_p$. Let $S = S_1 + \dots + S_n$, and for $1 \le i \le n$, set
\[
\hat S_i=S_1+\cdots+S_{i-1}+S_{i+1}+\cdots +S_n.
\]
Then we have
\[
|S| \le \Big(\prod_{i=1}^n |\hat S_i|\Big)^{\frac 1{n-1}}.
\]
\end{theorem}

Applying this to Lemma \ref{hp}, we obtain the following generalization. 

\begin{corollary}\label{directsum} Let $k$ be a divisor of $p-1$, $k \neq 1$, and $A_k$ be the group of $k$th powers in $\Zp^*$. If $A_k=S_1+\cdots +S_n$ for some nonempty subsets $S_i$ of $\mathbb Z_p$, then $|A_k|=\prod_{i=1}^n|S_i|$.
\end{corollary}

\begin{proof} Lemma \ref{hp} proves this for $n=2$.  Suppose now that $n \ge 3$ and that $A_k=S_1+\cdots +S_n = S_i+\hat S_i$.  Applying Lemma \ref{hp} to the decomposition
$
A_k=S_i+\hat S_i,
$
we obtain $|A_k|=|S_i||\hat S_i|$ for $1 \le i \le n$. Then by Theorem \ref{ruz} we have
\[
|A_k|^{n-1} \le \prod_{i=1}^n |\hat S_i|= \prod_{i=1}^n \frac {|A_k|}{|S_i|} = \frac {|A_k|^n}{\prod_{i=1}^n |S_i|}.
\]
This yields $\prod_{i=1}^n |S_i| \leq |A_k| $, and it is trivial that $|A_k| \le \prod_{i=1}^n |S_i|$. 
\end{proof}

Combining this with Lemma \ref{size2lem}, we arrive at the following corollary, which sets up the proof of Theorem \ref{mainthm}. 
\begin{corollary}\label{2by2corr}
Let $p$ be a prime and $A_k$ be the subgroup of $k$th powers in $\Zp^*$ for some $k|(p-1)$. Suppose $A_k$ is a generalized arithmetic progression and $k>1$. Then, $|A_k| = 2^n$ for some positive integer $n$, and $A_k$ is a proper generalized arithmetic progression of dimension $n$ with $L_i= 2$ uniformly.
\end{corollary}

\begin{proof}
Suppose $A_k$ is a GAP. Lemma \ref{size2lem} tells us that for some positive integer $n$, $A_k = S_1 + \dots + S_n$ for subsets $S_i \subseteq \Zp$ with $|S_i| = 2$. It follows from Corollary \ref{directsum} that such a decomposition must be a direct sum; that is, $|A_k|=\prod_{i=1}^n|S_i| = 2^n$. By definition, this means $A_k$ is a proper GAP of dimension $n$. 

\end{proof}
Note that this recovers the result of Chowla, Mann, and Straus \cite{cho} on arithmetic progressions stated in the introduction. For $k = 1$, $A_k = \Zp^*$ is trivially an arithmetic progression. For $k>1$, it follows from Corollary \ref{2by2corr} that only subgroups of order $2$ can form a GAP of dimension $1$.

\section{Proof of theorem 1.1}

We begin by considering $k = 2$, where $A_k$ is the group of quadratic residues. 

\begin{lemma}[\cite{che}; Lemma 2.6]\label{chenyan}
If $k=2$ and $A_k = U+V$ with $|V|\geq |U|\geq 2$, then $|U|\geq 5$.  
\end{lemma} 
Then by Corollary \ref{2by2corr} it follows that for $t>2$, no subgroup of quadratic residues can be expressed as a generalized arithmetic progression. 

For $k > 2$, we consider subgroups in two primary cases, $t < 64$ and $t \geq 64$, or equivalently, $n < 6$ and $n \geq 6$. For the first case, we look at the doubling properties of subgroups. For the second, we employ classical estimates on the number of solutions to the equation $x^k+y^k = c z^k$. Lemma 5.2 below is used in both cases. 

For any $c \in \Zp$, let $r(c)$ denote the representation function which counts the number of ways $c$ can be written as a sum of $2$ elements of $A_k$:
\[
r(c):= \#\{(x,y)\in A_k \times A_k: x+y=c \}.
\]

\begin{lemma}\label{repslem}
Suppose that $A_k=S_1+\cdots +S_n$ for subsets $S_i$ each of cardinality 2. Then, $r(c) \geq \frac{t}{2}$ for some $c \neq 0$. 
\end{lemma}

\begin{proof} Suppose that $A_k=S_1+\cdots +S_n$ for subsets $S_i$, each of cardinality 2.  By translation we can assume $0 \in S_i$ for the first $n-1$ sets, say $S_i=\{0,d_i\}$ and $S_n=\{a, a+d_n\}$ for some $d_i \in \Zp^*$,  $1 \le i \le n$, and $a \in \mathbb Z_p$.  Let $c=2a+\sum_{i=1}^{n-1}d_i$. We claim $c = -d_n \neq 0$. Indeed,  the sum of the elements in $A_k$ is zero, so
\[
0 = \sum_{x_1=0}^1\cdots \sum_{x_n=0}^1 (d_1x_1+\cdots +d_{n}x_{n}+a)=
2^{n-1}(2a+d_1+\cdots +d_n)=2^{n-1}(c+d_n).
\]

For $x=a +\sum_{i=1}^{n-1} x_i d_i$ with $x_i \in \{0,1\}$, $1 \le i \le n-1$, and  $y=a+ \sum_{i=1}^{n-1} (1-x_i)d_i$, we have $x+y=c$. Moreover, the pairs $x,y$ are uniquely determined by the choice of the $x_i$ as by Corollary \ref{directsum}, the decomposition of $A_k$ must be a direct sum. Thus $r(c) \ge 2^{n-1}=\frac{t}{2}$.
\end{proof}

Now we examine subgroups with $k>2$ and $t<64$. In the introduction, we showed that subgroups of order $2$ and $4$ will always be a GAP, and by Corollary \ref{2by2corr} we know that no subgroup with $k>1$ and $t \not \in \{2,4, 8,16,32\}$ can be a GAP. We restrict our attention here to $t\in \{8,16,32\}$, and use $\phi(t)$ to denote the Euler totient function. For $p<\frac{1}{2}3^{\phi(t)+1}$, a direct computation confirms that $r(c) < \frac{t}{2}$ holds for all such subgroups and $c \in \Zp^*$. For $p>\frac{1}{2}3^{\phi(t)+1}$ we have the following lemma.

\begin{lemma} [\cite{co}; Lemma 4]
For a prime number $p$ and $k|(p-1)$, let $A_k$ denote the multiplicative subgroup of $k$th powers in $\Zp^*$ with order $t$. If $t$ is even and $p>\frac{1}{2}3^{\phi(t)+1}$, then $\frac{|A_k+A_k|}{|A_k|} = \frac{t}{2} + \frac{1}{t}$.
\end{lemma} 
It follows that for $t = 2^n$ and $p>\frac{1}{2}3^{\phi(t)+1}$,  $\frac{|A_k+A_k|}{|A_k|} >  2^{n-1}$. 
Now suppose that $S$ is a GAP of dimension $n$ with $|S| = 2^n$, that is, \[ 
S=a +d_1[0,1]+ \dots + d_n[0,1]. 
\] The sumset $S+S$ is of the form: \[S+S = 2a + d_1[0,2]+ \dots + d_n[0,2], \] so plainly $|S+S| \leq 3^n$. Then for $n\in \{3,4,5\}$ we have: \[\frac{|S+S|}{|S|} \ \leq  \ \left( \frac{3}{2} \right)^n  < \ 2^{n-1}.\] Thus, if $k>2$ and $t<64$, $A_k$ cannot be a GAP unless $t = 2$ or $4$.

Finally, we examine the case with $k>2$ and $t\geq 64$. Consider the projective Fermat curve defined by $x^k+y^k = c z^k$ with $c\neq 0$, which has  genus $g = \frac{1}{2}(k-1)(k-2)$. The Hasse-Weil estimate \cite{da} for the number $N(c)$ of projective points over $\Zp$ on such a curve is $|N(c)-(p+1)| \le 2g \sqrt{p}$, which gives $N(c) \leq (k-1)(k-2)\sqrt{p} +p +1$. Then, because the map $(x,y,1)\mapsto (x^k,y^k)$ from points on the curve with $xyz \neq 0$ to solutions of $u+v = c$ with $(u,v) \in A_k \times A_k$ is $k^2$-to-$1$, we have that for all nonzero $c$,
\begin{equation}\label{HV}
r(c)  \le  \frac {N(c)}{k^2} \le \frac {p+1}{k^2} +\frac {(k-1)(k-2)}{k^2} \, \sqrt{p}.
\end{equation}
For small subgroups, we use the following refinement obtained by Mattarei \cite{ma}.
\begin{lemma} [\cite{ma}; Corollary]\label{mat}
If $k \ge 4$ and $t \le \frac 14 k^3$, then for any nonzero $c$  we have $r(c) \le 3\cdot 2^{-\frac 23} t^{\frac 23}.$
\end{lemma}

\begin{lemma} If $k>2$ and $t=2^{n}$ with $n \ge 6$, then $A_k$ cannot be expressed as a sum of sets each of cardinality $2$. That is, if $A_k=S_1+\cdots +S_n$  for some subsets $S_i$ of $\mathbb Z_p$, then $|S_i| \ge 3$ for some $i$. 
\end{lemma}

\begin{proof}
Suppose that $A_k=S_1+\cdots +S_n$ for subsets $S_i$ with $|S_i| = 2$.
If $t>\frac 14 k^3$ we obtain from Lemma \ref{repslem} and \eqref{HV} that for some $c\in \Zp^*$,
\begin{equation}\label{lemHV}
\frac t2 \le r(c) \le \frac tk+\frac 2{k^2} +\left(1-\frac 1k\right)\left(1-\frac 2k\right) \sqrt{kt+1},
\end{equation}
from which we get
\[
t\left(\frac 12-\frac 1k\right) \le \frac 2{k^2} + \left(1-\frac 1k\right)\left(1-\frac 2k\right) \sqrt{kt+1},
\]
that is, 
\[
t \le \frac 4{k(k-2)} +2\left(1-\frac 1k\right)\sqrt{kt+1}.
\]
For $k \ge 3$ and $t \ge 2^6$, the latter quantity is less than $2\sqrt{kt}$. 
Thus  $t<2\sqrt{kt}$, and so $t<4k$. Combining this with the assumption $t>\frac 14 k^3$, we get $\frac 14k^3<4k$, implying that $k<4$.  For $k=3$, \eqref{lemHV} implies $t \le 8$ contrary to assumption.

Suppose now that  $t \le \frac 14 k^3$.   In particular, $k \ge 4^{\frac 13}t^{\frac 13} \ge 4^{\frac 13} 64^{\frac 13}>4$. It follows from  Lemmas \ref{repslem} and \ref{mat} that for some $c \neq 0$, we have $\frac 12 t \le r(c) \leq 3\cdot 2^{-\frac 23} t^{\frac 23}$ and so $t \le 54$ contrary to assumption.
\end{proof}
\noindent Thus, no subgroup with $k>2$ and $t \geq 64$ can be a GAP, completing the proof of Theorem \ref{mainthm}.

\section{Another proof of Theorem 1.1}
Upon completion of this manuscript, it was brought to our attention that a recent preprint of Kalmynin \cite{ka} provides an alternate direct path to proving Theorem \ref{mainthm}. Of particular relevance is the following result.

\begin{lemma}[\cite{ka}; Corollary 1]\label{kalm}
For all primes $p$, no proper subgroup $A_k$ of $\Zp^*$ can satisfy $A_k = S_1+S_2+S_3$ with $|S_1|,|S_2|,|S_3| > 1$.
\end{lemma}

For a prime $p$, suppose that a proper subgroup $A_k$ of $\Zp^*$ is a generalized arithmetic progression. Then by Lemma \ref{size2lem}, this means that for some positive integer $m$ and subsets $S_i \subseteq \Zp$ with $|S_i| = 2$,
\[
A_k = S_1 + \dots + S_m = S_1+S_2+( S_3+ \dots + S_m).
\]

From Lemma \ref{kalm}, we immediately deduce that such a decomposition is possible only if $m\leq 2$. Together with Lemma \ref{hp}, it follows that if a proper subgroup $A_k$ of $\Zp^*$ forms a GAP, then $|A_k| = 2$ or $4$.

\subsection*{Acknowledgements}

We thank anonymous referees for their valuable suggestions, and for bringing Kalmynin's work to our attention. We also thank Craig Spencer for many helpful conversations.

\end{document}